\documentclass[smallextended,numbook,runningheads]{svjour31}     % onecolumn (second format)
\smartqed  % flush right qed marks, e.g. at end of proof
\usepackage{graphicx}
\usepackage{mathptmx}      % use Times fonts if available on your TeX system

\usepackage{latexsym,bm,amssymb,amsfonts,amsbsy,amsmath,graphics,color}

\newtheorem{theo}{Theorem}
\newtheorem{lem}{Lemma}
\newtheorem{cor}{Corollary}
\newtheorem{rem}{Remark}
\newtheorem{defn}{Definition}
\def\RR{\mathbb R}
\def\NN{\mathbb N}
\def\pmatrix{ \left( \begin{array} }
\def\endpmatrix{ \end{array} \right) }
\def\aa{\alpha}
\def\cc{\gamma}
\def\bff{{\bf f}}
\def\bfg{{\bf g}}
\def\bfI{{\bf I}}

\def\bfgam{\bm{\gamma}}

\def\bfzeta{\bm{\zeta}}
\def\bfp{{\bf p}}
\def\bfu{{\bf u}}
\def\bfy{{\bf y}}
\def\bfw{{\bf w}}
\def\bfz{{\bf z}}
\def\bfo{{\bf 0}}

\def\no{\noindent}
\def\diag{{\rm diag}}
\def\proof{\underline{Proof}\quad}
\def\QED{~\mbox{$\Box$}}

\def\phi{\varphi}

\def\hc{\tau}

\def\hy{w}

\def\hbfy{\hat{\bfy}}
\def\hbff{\hat{\bff \hphantom{.}}\hspace*{-.1cm}}
\def\hbfg{\hat{\bfg}}
\def\hbfp{\hat{\bfp}}
\def\hbfI{\hat{\bfI}}
\def\P{{\cal P}}
\def\hP{\hat{\P}}
\def\hG{\hat{G}}
\def\cI{{\cal I}}
\def\hcI{\hat{\cI}}

\def\bhzeta{\hat{\bfzeta}}
\def\hbfu{\hat{\bfu}}
\def\hbfz{\hat{\bfz}}

\def\ind{\mbox{\em ind}}

\journalname{~}
\begin{document}

\title{Analisys of Hamiltonian Boundary
Value Methods (HBVMs): a class of energy-preserving Runge-Kutta
methods for the numerical solution of polynomial Hamiltonian
systems\thanks{Work developed within the project ``Numerical
methods and software for differential
equations''.}} %%\subtitle{Hamiltonian BVMs}

\titlerunning{Hamiltonian BVMs}        % if too long for running head

\author{L.\,Brugnano \and
F.\,Iavernaro \and
 D.\,Trigiante }

\authorrunning{B.\,I.\,T.} % if too long for running head

\institute{Luigi Brugnano \at Dipartimento di Matematica
``U.\,Dini'', Universit\`a di Firenze, Viale Morgagni 67/A, 50134
Firenze (Italy).\\ \email{luigi.brugnano@unifi.it} \and Felice
Iavernaro \at Dipartimento di Matematica, Universit\`a di Bari,
Via Orabona 4,  70125 Bari (Italy).\\ \email{ felix@dm.uniba.it}
\and Donato Trigiante \at Dipartimento di Energetica
``S.\,Stecco'', Universit\`a di Firenze, Via Lombroso 6/17, 50134
Firenze (Italy).\\ \email{ trigiant@unifi.it} }

\date{Received: September 4, 2009 / Revised: December 2, 2009}

\maketitle

\begin{abstract}
One main issue, when numerically integrating autonomous
Hamiltonian systems, is the long-term conservation of some of its
invariants, among which the Hamiltonian function itself. For
example, it is well known that classical symplectic methods can
only exactly preserve, at most, quadratic Hamiltonians. In this
paper, a new family of methods, called {\em Hamiltonian Boundary
Value Methods (HBVMs)}, is introduced and analyzed. HBVMs are able
to {\em exactly} preserve, in the discrete solution, Hamiltonian
functions of polynomial type of arbitrarily high degree. These
methods turn out to be symmetric, precisely $A$-stable, and can
have arbitrarily high order. A few numerical tests confirm the
theoretical results.

\end{abstract}

\section{Introduction}\label{intro}
\vspace{-.5cm}
\begin{flushright}
  \textsf{\em Omne ignotum pro magnifico}\\
  \textsf{Tacitus, \em Agr. 30}\\
  \textsf{\em (sed interdum exitus mathematicae\\
  investigationis vere magnifici sunt)}
\end{flushright}

The numerical solution of Hamiltonian problems is a relevant issue
of investigation since many years: we refer to the recent
monographs \cite{HLW,LR} for a comprehensive description of this
topic, and to the references therein.

In a certain sense, the use of a numerical method acts as
introducing  a small perturbation in the original system which, in
general, destroys all of its first integrals. The study of the
preservation of invariant tori in the phase space of nearly
integrable Hamiltonian systems has been a central theme in the
research since the pioneering work of Poincar\'e, the final goal
being to asses the stability of the solar system. From a numerical
point of view, results in this respect are still poor, and this is
justified by considering the delicacy of the problem: as testified
by KAM theory, even small Hamiltonian perturbations of completely
integrable systems, do not prevent the disappearance of most of
the tori, unless a Diophantine condition on the frequencies of the
unperturbed system is satisfied.

At the times when research on this topic was started, there were
no available numerical methods possessing such conservation
features. A main approach to the problem was the devising of
symplectic methods. However, though the numerical solution
generated by symplectic (and/or symmetric) methods  shows some
interesting long-time behavior (see, for example, \cite[Theorems
X.3.1 and XI.3.1]{HLW}), it was observed that symplecticity alone
can only assure, at most, the conservation of quadratic
Hamiltonian functions, unless they are coupled with some
projection procedure. In the general case, conservation cannot be
assured, even though a quasi-preservation can be expected for
reversible problems, when symmetric methods are used (see, e.g.,
\cite{BT2}). On the other hand, a numerical ``drift'' can be
sometimes observed in the discrete solution \cite{Faou}. One of
the first successful attempts to solve the problem of loss of
conservation of the Hamiltonian function by the numerical
solution, is represented by \textit{discrete gradient methods}
(see \cite{MQR} and references therein). Purely algebraic
approaches have been also introduced (see, e.g., \cite{CFM}),
without presenting any energy-preserving method.

A further approach was considered in \cite{QL}, where the {\em
averaged vector field method} was proposed and shown to conserve
the energy function of canonical Hamiltonian systems. As was
recently outlined (see \cite{M2AN}), approximating the integral
appearing in such method by means of a quadrature formula (based
upon polynomial interpolation) yields a family of second order
Runge-Kutta methods. These latter methods represent an instance of
energy-preserving Runge-Kutta methods for polynomial Hamiltonian
problems: their first appearance may be found in \cite{IP1}, under
the name of {\em $s$-stage trapezoidal methods}. Additional
examples of fourth and sixth-order Runge-Kutta methods were
presented in \cite{IP2} and \cite{IT2}.

In \cite{IP1,IP2,IT2}, the derivation of such energy preserving
Runge-Kutta formulae relies on the  definition of the so called
``discrete line integral'', first introduced in \cite{IT1}.
However, a comprehensive analysis of such methods has not been
carried out so far, so that their properties were not known and,
moreover, their practical construction was difficult.

In this paper we provide such an analysis, which allows us to
derive symmetric methods, of arbitrarily high order, able to
preserve Hamiltonian functions of polynomial type, of any
specified degree. Such methods are here named {\em Hamiltonian
Boundary Value Methods (HBVMs)}, since the above approach has
been, at first, studied (see, e.g., \cite{IP2,IT2}) in the
framework of {\em block Boundary Value Methods}. The latter are
block one-step methods \cite{BT1}. However, for the sake of
clarity, and later reference, the equivalent Runge-Kutta
formulation of HBVMs will be here also considered.

In the remaining part of this section, we introduce the background
information concerning the approach. Let then

\begin{equation}\label{hampro}
y' = J\nabla H(y), \qquad y(0) = y_0\in\RR^{2m},
\end{equation}

\no be a Hamiltonian problem in canonical form, where, by setting
as usual $I_m$ the identity matrix of dimension $m$,
\begin{equation}\label{J} J=\pmatrix{rr}
&I_m\\-I_m\endpmatrix,\end{equation}

\no and where the Hamiltonian function, $H(y)$, is a polynomial of
degree $\nu$. It is well known that, for any $y^* \in \RR^{2m}$,
\begin{equation}\label{Hy}
H(y^*) - H(y_0) =\int_{y_0\rightarrow y^*} \nabla H(y)^Tdy
       = \int_0^1 \sigma'(t)^T\nabla H(\sigma(t))dt,
\end{equation}

\no where $\sigma: [0,1] \rightarrow R^{2m}$ is any smooth
function such that $$\sigma(0) = y_0, \qquad\sigma(1) = y^*.$$

In particular, over a trajectory, $y(t)$, of
(\ref{hampro}), one has
\begin{eqnarray*}
H(y(t)) - H(y_0)  &=& \int_0^t \nabla H(y(\tau))^T y'(\tau)d\tau \\
\nonumber &=& \int_0^t \nabla H(y(\tau))^T J \nabla
H(y(\tau))d\tau = 0,\end{eqnarray*}

\no due to the fact that matrix $J$ in (\ref{J}) is
skew-symmetric.

Here we consider the case where $\sigma(t)$ is a polynomial of
degree $s$ yielding an approximation to the true solution $y(t)$
in the time interval $[0, h]$ which, without loss of generality,
is hereafter normalized to $[0,1]$. More specifically, given the
$s+1$ abscissae
\begin{equation}\label{ci}
0=c_0<c_1<\dots<c_s = 1,
\end{equation}
and the approximations $y_i \approx y(c_i)$, $\sigma(t)$ is meant
to be defined by  the interpolation conditions
\begin{equation}\label{yi}
\sigma(c_i) = y_i, \qquad i=0,\dots,s.
\end{equation}

\no Actually, the approximations $\{y_i\}$ will be unknown, until
the new methods will be fully derived.

A different, though related concept, is that of collocating
polynomial for the problem, at the abscissae (\ref{ci}). Such a
polynomial is the unique polynomial $u(t)$, of degree $s+1$,
satisfying
\begin{equation}\label{coll}
u(c_0) = y_0, \qquad\mbox{and}\qquad u'(c_i) = J\nabla H(u(c_i)),
\qquad i=0,\dots,s.
\end{equation}

It is well known that (\ref{coll}) define a Runge-Kutta {\em
collocation method}. Moreover, the set of abscissae (\ref{ci})
defines a corresponding quadrature formula with weights
\begin{equation}\label{bi}b_i = \int_0^1 \prod_{j=0,\,j\ne
i}^s \frac{t-c_j}{c_i-c_j}dt, \qquad i=0,1,\dots,s, \end{equation}

\no which has degree of precision ranging from $s$ to $2s-1$,
depending on the choice of the abscissae (\ref{ci}). In
particular, the highest precision degree is obtained by using the
Lobatto abscissae, which we shall consider in the
sequel.\,\footnote{\, Different choices of the abscissae will be
the subject of future investigations.} The underlying collocation
method has, then, order $2s$.

\begin{rem}\label{Rs}
Choosing a Gauss distribution of the abscissae $\{c_i\}$ raises
the degree of precision of the related  quadrature formula to
$2s+1$. In such a case, it is interesting to observe that applying
(\ref{Hy}) along the trajectory $u(t)$ and exploiting the
collocation conditions \eqref{coll}, one gets
\begin{eqnarray}\label{intex}
H(u(c_s))-H(y_0)&=&\int_0^1 u'(t)^T\nabla H(u(t))dt\\ &=&
\sum_{i=0}^s b_i u'(c_i)^T\nabla H(u(c_i)) +R_s = R_s,\nonumber
\end{eqnarray}

\no where $R_s$ is the error in the approximation of the line
integral. Therefore, $H(u(c_s))=H(y_0)$ if and only if $R_s=0$,
which is implied by assuming that the quadrature formula with
abscissae $\{c_i\}$ and weights $\{b_i\}$ is exact when applied to
the integrand $u'(t)^T \nabla H(u(t))$. However, since the
integrand has degree
$$ s + (\nu-1)(s+1) = \nu (s+1) -1,$$

\no it follows that the maximum allowed value for $\nu$ is 2.
Indeed, it is well known that quadratic invariants are preserved
by symmetric collocation methods.  On the other hand, when
$\nu>2$,  in general $R_s$ does not vanish, so that $H(u(c_s))\ne
H(y_0)$.
\end{rem}

The above remark gives us a hint about how to approach the
problem. Note that in \eqref{intex} demanding that each term of
the sum representing the quadrature formula is null (i.e., the
conditions (\ref{coll})), is an excessive requirement to obtain
the conservation property, which causes the observed low degree of
precision. A weaker assumption, that would leave the result
unchanged, is to relax conditions \eqref{coll} so as to devise a
method whose induced quadrature formula, evaluated on a suitable
line integral that links two successive points of the numerical
solution, is exact and, at the same time, makes the corresponding
sum vanish, without requiring that each term is
zero.\footnote{\,More precisely, in the new methods, conditions
\eqref{coll} will turn out to be replaced by relations of the form
$\sigma'(c_i)=\sum_j \beta_{ij} J \nabla H(\sigma(c_j))$, which
resemble a sort of \textit{extended collocation condition} (see
also \cite[Section\,2]{IT2}) since $\sigma'(c_i)$ brings
information from the global behavior of the problem in the time
interval $[0, h]$ (see \eqref{expan}--\eqref{rhs} in
Section~\ref{colloc} and the analogues in Section \ref{derive}).}

If we use $\sigma(t)$ instead of $u(t)$, the integrand function
in \eqref{Hy} has degree $\nu s -1$ so that, in order for the
quadrature formula to be exact, one would need say, $k+1$ points,
where
\begin{equation}\label{kappa}
k=\left\lceil\frac{\nu s}2\right\rceil,
\end{equation}

\no if the corresponding Lobatto abscissae are used. Of course, in
such a case, the vanishing of the quadrature formula is no longer
guaranteed by conditions \eqref{coll} and must be obtained  by a
different approach. For this purpose, let
\begin{equation}\label{erre}
r = k-s,
\end{equation}

\no  be the number of the required additional points, and let
\begin{equation}\label{hci}
0<\hc_1<\dots<\hc_r<1,\end{equation}

\no be $r$ additional abscissae distinct from (\ref{ci}).
Moreover, let us define the following {\em silent stages} \cite{IT2},
\begin{equation}\label{hyi} \hy_i \equiv \sigma(\hc_i), \qquad
i=1,\dots,r.\end{equation}

Consequently, the polynomial $\sigma(t)$, which interpolates the
couples $(c_i,y_i)$, $i=0,1,\dots,s$, also interpolates the
couples $(\hc_i,\hy_i)$, $i=1,\dots,r$. That is, $\sigma(t)$
interpolates at $k+1$ points, even though it has only degree $s$.
If we define the abscissae
\begin{equation}\label{ti}
\{ t_0 < t_1 < \dots < t_k \} = \{ c_i \} \cup \{ \hc_i\},
\end{equation}

\no and dispose them according to a Lobatto distribution in
$[0,1]$ in  order to get a  formula of degree $2k$, we have that
\begin{equation}
\label{exactness}
\int_0^1 \sigma'(t)^T\nabla H(\sigma(t))dt = \sum_{i=0}^k b_i
\sigma'(t_i)^T\nabla H(\sigma(t_i)),
\end{equation}

\no and, consequently, the conservation condition becomes
\begin{equation}\label{intexk}
\sum_{i=0}^k b_i \sigma'(t_i)^T\nabla H(\sigma(t_i))=0,\end{equation}

\no where, now,
\begin{equation}\label{bik}
b_i = \int_0^1 \prod_{j=0,\,j\ne i}^k \frac{t-t_j}{t_i-t_j}dt,
\qquad i=0,1,\dots,k.
\end{equation}

The left-hand side of \eqref{intexk} is called  ``discrete line
integral'' because, as will be clear in the sequel, the choice of
the path $\sigma(t)$ is dictated by the numerical method by which
we will solve problem \eqref{hampro} (see \cite{IT2} for details).

With these premises, in Section~\ref{derive}, we devise such a
method, able to fulfill (\ref{intexk}),  after having set some
preliminary results in Section~\ref{colloc}. Before that, in
Section~\ref{Legendre} we state a few facts and notations
concerning the shifted Legendre polynomials, which is the
framework that we shall use to carry out the analysis of the
methods. A few numerical tests are then reported in
Section~\ref{tests} and, finally, a few conclusions are given in
Section~\ref{fine}. For sake of completeness, some properties of
shifted Legendre polynomials are listed in the Appendix.

\section{Preliminary results and notations}\label{Legendre}
The {\em shifted Legendre polynomials}, in the interval $[0,1]$,
constitute a family of polynomials, $\{P_n\}_{n\in\NN}$, for which
a number of known properties, named {\bf P1}--{\bf P12}, are
reported in the Appendix. We now set some notations and results,
to be used later.

With reference to the abscissae (\ref{ci}), let:
\begin{equation}\label{pj}\bfp_j =
\pmatrix{c}P_j(c_1)\\ \vdots\\ P_j(c_s)\endpmatrix, \qquad \hbfp_j
= \pmatrix{c} P_j(c_0)\\ \bfp_j \endpmatrix, \qquad j=0,\dots,s,
\end{equation}
\begin{equation}\label{P}
\P_j = \pmatrix{ccc} \bfp_0 & \dots &\bfp_j\endpmatrix\in
\RR^{s\times j+1},\quad\hP_j = \pmatrix{ccc} \hbfp_0 & \dots
&\hbfp_j\endpmatrix\in \RR^{s+1\times j+1},\end{equation}
\begin{equation}\label{Ij}\bfI_j =
\pmatrix{c}\int_{0}^{c_1}P_j(x)dx\\ \vdots\\
\int_{0}^{c_s}P_j(x)dx\endpmatrix,\quad\hbfI_j =
\pmatrix{c}\int_{0}^{c_0} P_j(x)dx\\ \bfI_j\endpmatrix \equiv
\pmatrix{c}0\\ \bfI_j\endpmatrix, \quad j=0,\dots,s.
\end{equation}

\begin{rem}\label{RIs} Observe that, from {\bf P11}, one obtains:
 \begin{equation}\label{Is}
  \bfI_s = \bfo.
 \end{equation}
\end{rem}

\no Furthermore, we set:
\begin{equation}\label{I}\cI_j = \pmatrix{ccc} \bfI_0 & \dots \bfI_j
\endpmatrix\in\RR^{s\times j+1}, \quad \hcI_j = \pmatrix{ccc} \hbfI_0 & \dots \hbfI_j
\endpmatrix\in\RR^{s+1\times j+1},
\end{equation}
\begin{equation}\label{DOm}
D_j=\pmatrix{cccc}1\\ &3\\&&\ddots\\&&&2j-1\endpmatrix\in\RR^{j\times j},
\quad \Omega = \pmatrix{ccc}b_0\\ &\ddots\\ &&b_s\endpmatrix,
\end{equation}
and  \begin{equation}\label{G}G_j = \pmatrix{rrrrr}
1 &-1\\
1 &0 &\ddots\\
  &1 &\ddots &-1\\
  &  &\ddots &0\\
  &  &     &1\endpmatrix\in\RR^{j+1\times j}.\end{equation}
By virtue of {\bf P3} and {\bf P9}, we deduce that
\begin{equation}\label{orto}
\hP_{j-1}^T \Omega \hP_j = \left[D_j^{-1}~\bfo\right],\qquad
j=1,\dots,s,
\end{equation}
and
\begin{equation}\label{integ}
\hcI_{j-1} = \frac{1}2\hP_j G_j D_j^{-1}, \qquad \cI_{j-1} =
\frac{1}2 \P_j G_j D_j^{-1},\qquad j=1,2,\dots.
\end{equation}

\begin{lem}\label{gram}
The matrix $\hP_s = \pmatrix{ccc} \hbfp_0& \dots &\hbfp_s\endpmatrix\in\RR^{s+1\times s+1}$ is nonsingular. Moreover
\begin{equation}\label{intgram}
\hcI_s = \pmatrix{cc} \bfo^T &0\\ \cI_{s-1} &\bfo\endpmatrix =
\hP_s \hG_s \equiv \hP_s \pmatrix{rrrrrr}
1 &-1     & 0 &\dots &0 \\
1 &0      &\ddots &\ddots &\vdots\\
  &\ddots &\ddots &-1 &0 &\\
  &       &\ddots &0 &0\\
  &       &     &1 &0\endpmatrix\in\RR^{s+1\times s+1},\end{equation}
  with $\cI_{s-1}\in\RR^{s\times s}$ a nonsingular matrix.
  \end{lem}

\proof $\hP_s$ is the transpose of the Gramian matrix defined by
the linearly independent polynomials $P_0(c),\dots,P_s(c)$ at the
distinct abscissae $c_0,\dots,c_s$ and is, therefore, nonsingular.
The structure of $\hcI_s$ follows from (\ref{Is}). The matrix
$\cI_{s-1}$ is nonsingular since, from (\ref{intgram}), $\hP_s$ is
nonsingular, and ${\rm rank}(\hG_s)=s$.\QED

\section{Matrix form of collocation methods}\label{colloc}
In this section we deliberately do not care of the exactness of
the discrete line integral, as stated by \eqref{exactness}, and in
fact we choose $k=s$ (and hence $t_i=c_i$, $i=0,\dots,s$). We show
that imposing the vanishing of the discrete line integral
(condition \eqref{intexk}) leads to the definition of the
classical Lobatto IIIA methods. The reason why we consider this
special situation is that the technique that we are going to
exploit is easier to be explained, but at the same time is
straightforwardly generalizable to the case $k>s$. As a
by-product, we will gain more insight about the link between the
new methods and the Lobatto IIIA class. For example, we will
deduce that Lobatto IIIA methods (and, in general, all collocation
methods) may be defined by means of a polynomial $\sigma(t)$ of
degree not larger than that of the collocation polynomial $u(t)$
(indeed, in the present case, $\deg \sigma(t) = \deg u(t)-1$).

To begin with, let us consider the following expansion of
$\sigma'(c)$:
\begin{equation}\label{expan}
\sigma'(c) = \sum_{j=0}^{s-1} \gamma_j P_j(c),
\end{equation}

\no where the  (vector) coefficients $\gamma_j$ are to be
determined. Then, (\ref{intexk}) becomes
\begin{equation}\label{orto0}
\sum_{j=0}^{s-1} \gamma_j^T \sum_{i=0}^s b_i P_j(c_i)\nabla
H(\sigma(c_i))=0,
\end{equation}

\no which will clearly hold true, provided that the following set
of orthogonality conditions are satisfied:
\begin{equation}\label{ortcon}
\gamma_j = \eta_j\sum_{i=0}^s b_i P_j(c_i) J\nabla H
(\sigma(c_i)), \qquad j=0,\dots,s-1,
\end{equation}

\no where $\{\eta_j\}$ are suitable scaling factors. We now impose
that the polynomial
\begin{equation}\label{expan1}
\sigma(c) = y_0 + \sum_{j=0}^{s-1} \gamma_j \int_0^c P_j(x)\,dx
\end{equation}

\no satisfies (\ref{yi}). By setting
\begin{equation}\label{y}
\bfgam=\pmatrix{c}\cc_0\\ \vdots \\ \cc_{s-1}\endpmatrix,
\quad e = \pmatrix{c}1\\ \vdots\\1\endpmatrix\in\RR^s, \quad
\bfy = \pmatrix{c}y_1\\ \vdots\\y_s\endpmatrix, \quad \hbfy =
\pmatrix{c}y_0\\ \bfy\endpmatrix,
\end{equation}

\no one obtains (see (\ref{Ij})--(\ref{integ}))
\begin{equation}\label{inter}
\cI_{s-1}\otimes I_{2m}\bfgam = \left( \frac{1}2
\P_sG_sD_s^{-1}\right)\otimes I_{2m} \bfgam =\bfy -e\otimes y_0.
\end{equation}

\no Consequently,
\begin{equation}\label{gammay}
\bfgam = \left[2D_s (\P_sG_s)^{-1} \pmatrix{cc}-e
&I_s\endpmatrix\right]\otimes I_{2m}\, \hbfy.
\end{equation}

\no On the other hand, the vector form of relations (\ref{ortcon}) reads
\begin{equation}\label{rhs} \bfgam= \left(\Gamma \hP_{s-1}^T
\Omega\right) \otimes I_{2m} \hbff,
\end{equation}

\no where $\Gamma = \diag( \eta_1,\dots,\eta_s)\in\RR^{s\times s}$ and
\begin{equation}\label{ff}\hbff = \pmatrix{ccc}f_0 &\dots&f_s\endpmatrix^T, \qquad
f_i = J\nabla H(\sigma(c_i)),\quad i=0,\dots,s.\end{equation}

\no Since $\Gamma$ contains free parameters, we set
\begin{equation}\label{Gamma}
\Gamma=D_s.
\end{equation}

\no Comparing \eqref{gammay} and \eqref{rhs},  we arrive at the
following  block method, where now $h$ denotes, in general, the
used stepsize,
\begin{equation}\label{AB}
A \otimes I_{2m}\, \hbfy = hB\otimes I_{2m} \,\hbff,
\end{equation}

\no with (see (\ref{integ}))
\begin{equation}\label{ABeq}
A = \pmatrix{cc} -e & I_s\endpmatrix, \qquad B =
\left(\frac{1}2\P_sG_s \hP_{s-1}^T\Omega\right) \equiv
\left(\cI_{s-1} D_s\hP_{s-1}^T\Omega\right).\end{equation}

The following noticeable result holds true.

\begin{theo}\label{ordth} Each row of the block method
(\ref{AB})-(\ref{ABeq}) defines a LMF of order $s+1$. The last row
corresponds to the Lobatto quadrature formula of order $2s$.
\end{theo}

\proof For the first part of the proof, it suffices to show that
the method is exact for polynomials of degree $s+1$. Clearly, it
is exact for polynomials of degree 0, due to the form of the
matrix $A$. We shall then prove that $A \hcI_s = B \hP_s$, that is
(see (\ref{Ij}), (\ref{I}), and (\ref{ABeq})), $\cI_s = B \hP_s$.
By virtue of (\ref{ABeq}), (\ref{orto}), and (\ref{Is}), we have
$$B\hP_s = \cI_{s-1}
D_s\hP_{s-1}^T\Omega\hP_s = \cI_{s-1} D_s\left[D_s^{-1}~
\bfo\right] = \left[\cI_{s-1}~\bfI_s\right] = \cI_s,$$ which
completes the first part of the proof. For the second part, one
has to show, by setting as usual $e_i$ the $i$th unit vector, that
$$e_s^TB = \pmatrix{ccc} b_0 & \dots &b_s\endpmatrix,$$

\no the vector containing the coefficients of the quadrature
formula.  From (\ref{ABeq}), exploiting property {\bf P7} (see
also (\ref{G})), we obtain
\begin{eqnarray*}
e_s^T B &=& \frac{1}2 e_s^T \P_s G_s \hP_{s-1}^T\Omega ~=~
\frac{1}2 \pmatrix{ccc} 1&\dots&1\endpmatrix G_s
\hP_{s-1}^T\Omega\\ &=& e_1^T \hP_{s-1}^T \Omega ~=~ \pmatrix{ccc}
1&\dots&1\endpmatrix \Omega = \pmatrix{ccc} b_0 &\dots&
b_s\endpmatrix. \QED
\end{eqnarray*}

As an immediate consequence, the following result follows.

\begin{cor}\label{collcor} The block method
(\ref{AB})-(\ref{ABeq}) collocates at the Lobatto abscissae
(\ref{ci}) and has global order $2s$.
\end{cor}

\proof The proof follows from known results about collocation
methods (see, e.g., \cite[Theorem~II.1.5]{HLW}).\QED

\begin{rem}\label{lobattoIIIA} In conclusion, the method corresponding
to the pencil $(A,B)$, as defined by (\ref{ABeq}), is nothing but the
Lobatto IIIA method of order $2s$.\end{rem}

\subsection{Link between $\sigma(c)$ and the collocation polynomial}

An important consequence of Theorem \ref{ordth} and Corollary
\ref{collcor} is that the Lobatto IIIA method of order $2s$ may be
also defined by means of an underlying polynomial, namely
$\sigma(c)$, of degree $s$  instead of $s+1$, as is the
collocation polynomial associated with the method \eqref{AB}.

The main aim of the present  subsection is to elucidate the
relation between these two polynomials. In what follows, we
deliberately ignore the result obtained in Theorem \ref{ordth} and
Corollary \ref{collcor}, so as to provide, among other things, an
alternative proof of part of the statements they
report.\footnote{\,The approach exploited in the proof of Theorem
\ref{ordth} turns out to be  crucial to deduce the new methods
presented in the next section.}

Let $u(c)$ be the  polynomial (\ref{coll}) (of degree $s+1$) that
collocates problem \eqref{hampro} at the abscissae (\ref{ci}). The
expansion of $u'(c)$ along the shifted Legendre polynomials basis
reads
\begin{equation}
\label{udot}
\displaystyle  u'(c) = \sum_{j=0}^{s} \zeta_j P_{j}(c).
\end{equation}

\no Consequently, by setting $$\hbfg = \pmatrix{c}g_0\\
\vdots\\g_s\endpmatrix, \quad g_i=J\nabla H (u(c_i)), \quad
\mbox{and}\quad \bhzeta \equiv \pmatrix{c} \bfzeta \\
\zeta_s\endpmatrix \equiv \pmatrix{c} \pmatrix{c}\zeta_0\\ \vdots
\\ \zeta_{s-1}\endpmatrix\\[1mm]  \zeta_{s}\endpmatrix,$$

\no one obtains that (\ref{coll}) may be recast in matrix notation
as $\hP_s \otimes I_{2m}  \bhzeta =  \hbfg$, or
\begin{equation}
\label{collocation1} \bhzeta = \hP_s^{-1} \otimes I_{2m}\, \hbfg.
\end{equation}

\no We get the expression of $u(c)$ by integrating both sides of
(\ref{udot}) on the interval $[0,\,c]$:
\begin{equation}
\label{uc}
u(c)=y_0 + \sum_{j=0}^{s-1} \zeta_j \int_0^c P_{j}(x)dx + \zeta_{s} \int_0^c P_{s}(x)dx.
\end{equation}

\no By virtue of property {\bf P11}, we get
\begin{equation}
\label{uci}
u(c_i)=y_0 + \sum_{j=0}^{s-1} \zeta_j \int_0^{c_i} P_{j}(x)dx, \qquad i=0,\dots,s.
\end{equation}

\no Setting $z_i=u(c_i)$,~$i=1,\dots,s$, $\bfz =
(z_1,\dots,z_s)^T$, and $\hbfz = (y_0, \bfz^T)^T$, allows us to
recast (\ref{uci}) in matrix form. This is done by exploiting a
similar argument used to get (\ref{inter}) starting from
(\ref{expan1}):
\begin{eqnarray}\nonumber
 A \otimes I_{2m}\hbfz &=&\bfz - e\otimes y_0 ~=~ {\cal I}_{s-1} \otimes
I_{2m} \bfzeta ~=~ \left(\frac{1}{2} \P_s G_s D_s^{-1}\right)
\otimes I_{2m} \bfzeta\\ &=& \left(\frac{1}{2} \P_s G_s \left[
D_s^{-1} ~~ \bfo \right] \right) \otimes I_{2m} \bhzeta.
\label{uci1}\end{eqnarray}

\no Inserting (\ref{collocation1}) into (\ref{uci1}), and
exploiting (\ref{orto}), yields
\begin{eqnarray*}
\lefteqn{A \otimes I_{2m}\hbfz}\\&=& \frac{1}{2} \P_s G_s \left[
D_s^{-1} ~~ \bfo \right] \hP_s^{-1}  \otimes I_{2m}  \hbfg =
\frac{1}{2} \P_s G_s \hP_{s-1}^{T} \Omega \otimes I_{2m}   \hbfg
\\&=& B \otimes I_{2m} \hbfg.
\end{eqnarray*}

\no Thus, the collocation problem  (\ref{coll}) defines the very
same method arising from the polynomial $\sigma(c)$ (see
(\ref{AB})--(\ref{ABeq})) with $h=1$. This implies that system
(\ref{AB}) is a collocation method defined on the Lobatto
abscissae $c_i$, $i=0,\dots,s$, (therefore, a Lobatto IIIA
method), and provides an alternative proof of Corollary
\ref{collcor}. In particular, we deduce that $$u(c_i) = y_i=
\sigma(c_i),\qquad i=0,\dots,s.$$

\no It follows that (\ref{uc}) becomes
\begin{equation}\label{usig}
u(c)=\sigma(c) + \zeta_{s} \int_0^c P_{s}(x)dx
\end{equation}

\no and, after differentiating,
\begin{equation}
\label{usigma}
u'(c)=\sigma'(c) + \zeta_{s} P_{s}(c).
\end{equation}

\no We can obtain the expression of the unknown $\zeta_{s}$ by
imposing a collocation condition at any of the abscissae $c_i$.
For example, choosing $c=c_s=1$, yields
\begin{equation}\label{zs}
\zeta_{s} = u'(1) - \sigma'(1) = f(y_s)-\sum_{j=0}^{s-1} \gamma_j
= f(y_s) - { e}^T \otimes I_{2m} \bfgam.
\end{equation}

\no This latter expression can be slightly simplified by
considering that:
\begin{itemize}
\item[(i)] $f(y_s)=(e^T \, 1)\hP_{s}^{-1} \otimes I_{2m} \hbff$,
which comes from the fact that the system  $\hP_{s}^T x =
\pmatrix{c} e
\\ 1 \endpmatrix$ has solution $x = e_{s+1}$ (the
nonsingularity of $\hP_{s}$ being assured by Lemma~\ref{gram});

\item[(ii)] from (\ref{gammay}) and (\ref{AB})--(\ref{ABeq}), one has
\begin{eqnarray*}\bfgam &=& (D_s \hP_{s-1}^T \Omega) \otimes I_{2m} \hbff~=~
(D_s
\hP_{s-1}^T \Omega \hP_{s} \hP_{s}^{-1}) \otimes I_{2m} \hbff\\
&=& (D_s (D_s^{-1}~ \bfo) \hP_{s}^{-1}) \otimes I_{2m} \hbff~=~
(I_s~ \bfo) \hP_{s}^{-1} \otimes I_{2m} \hbff.\end{eqnarray*}
\end{itemize}

\no Thus, from (\ref{zs}) we get
\begin{equation}
\label{zeta_last} \zeta_{s} = \left[(e^T~1) -(e^T~0)\right]
\hP_{s}^{-1} \otimes I_{2m} \hbff = e_{s+1}^T \hP_{s}^{-1} \otimes
I_{2m} \hbff.
\end{equation}

\no The remaining collocation conditions, $u'(c_i)=J\nabla
H(u(c_i)), i=0,\dots,s-1$, are clearly satisfied since the
collocation polynomial $u(c)$ is uniquely identified by the $s+2$
linearly independent conditions in (\ref{coll}). Nonetheless, they
can be easily checked after observing that, from (\ref{usig}),
(ii), and (\ref{zeta_last}),
$$
\bhzeta= \left( \begin{array}{c} \bfgam \\ \zeta_{s} \end{array}
\right) = \hP_{s}^{-1} \otimes I_{2m} \hbff.
$$

\no Therefore, from (\ref{udot}), (\ref{collocation1}), and
(\ref{usig}), one obtains,
$$
\hbfu' \equiv \left( \begin{array}{c} u'(c_0) \\ \vdots \\ u'(c_s)
\end{array} \right) = \hP_{s} \otimes I_{2m} \bfzeta = \hP_{s}\hP_{s}^{-1}
\otimes I_{2m}  \hbff = \hbff.
$$

\no That is (see (\ref{ff})),~  $u'(c_i) = J \nabla H(u(c_i))$,~
$i=0,\dots,s$.

\section{Derivation of the methods}\label{derive}
The arguments for deriving the methods in the general case where
$k\ge s$ (which makes the discrete line integral exact) are a
straightforward  extension of what stated above. In particular,
let us consider again the expansion (\ref{expan})--(\ref{expan1})
of the polynomial $\sigma(c)$. Then, condition (\ref{intexk}) can
be recast as (compare with (\ref{orto0}))
\begin{equation}\label{orto1}
\sum_{j=0}^{s-1} \gamma_j^T \sum_{i=0}^k b_i P_j(t_i)\nabla
H(\sigma(t_i))=0,
\end{equation}

\no which will clearly hold true, provided that the following set
of orthogonality conditions are satisfied (compare with
(\ref{ortcon}), see also (\ref{bik})):
\begin{equation}\label{ortcon1}
\gamma_j = \eta_j\sum_{i=0}^k b_i P_j(t_i) J\nabla H
(\sigma(t_i)), \qquad j=0,\dots,s-1,
\end{equation}

\no where $\{\eta_j\}$ are suitable scaling factors. According to
(\ref{Gamma}), we choose them as $\eta_j=2j+1$, $j=0,\dots,s-1$.
The vector $\bfgam$ (see (\ref{y})) is then obtained by imposing
that the polynomial $\sigma(c)$ in (\ref{expan1}) satisfies the
interpolation constrains (\ref{yi}) and (\ref{hyi}). In so doing,
one obtains a block method characterized by the pencil $(A,B)$,
where the two $k\times k+1$ matrices $A$ and $B$ are defined as
follows. In order to simplify the notation, we shall use a
``Matlab-like'' notation: let  $\ind_s\in\RR^{s+1}$ and
$\ind_r\in\RR^r$ be the vectors whose entries are the indexes of
the main abscissae $c_0<\dots<c_s$ in (\ref{ci}) and of the silent
ones $\hc_1<\dots<\hc_r$ in (\ref{hci}), respectively, within the
Lobatto abscissae $t_0<\dots<t_k$, as defined in (\ref{ti}). Then,
the orthogonality conditions (\ref{ortcon1}) will define the first
$s$ rows of $A$ and $B$\,\footnote{\,As a further convention, the
entries not explicitly set are assumed to be 0.} (compare with
(\ref{ABeq})):
\begin{equation}\label{hbvmks}
A(1:s,\ind_s) = \pmatrix{lr} -e & I_s\endpmatrix, \qquad B(1:s,:)
= \left(\cI_{s-1} D_s\bar\P^T\bar\Omega\right),\end{equation}

\no where (see (\ref{I})--(\ref{DOm}) and (\ref{ti}))
\begin{equation}\label{bps1}
\bar\P = \pmatrix{ccc} P_0(t_0)& \dots &P_{s-1}(t_0)\\
\vdots & &\vdots\\ P_0(t_k)&\dots &P_{s-1}(t_k)\endpmatrix\in
\RR^{k+1\times s},\end{equation}

\no and (see (\ref{bi})) \begin{equation}\label{barOm}
\bar\Omega = \pmatrix{ccc} b_0\\ &\ddots\\
&&b_k\endpmatrix\in\RR^{k+1\times k+1}.\end{equation}

\no On the other hand, the interpolation conditions for the silent
stages (\ref{hyi}) define the last $r$ rows of the matrix $A$ (the
corresponding rows of $B$ are obviously zero):
\begin{eqnarray}\label{hbvmks1}
A(s+1:k,\ind_r) &=& I_r, \\ \nonumber
A(s+1:k,\ind_s) &=& -\bar\cI_r\left[\cI_{s-1}^{-1}\pmatrix{lr} -e &
I_s\endpmatrix\right] - \bar{e}\cdot e_1^T,
\end{eqnarray}

\no where $I_r$ is the identity matrix of dimension $r$, $\bar{e} =
(1,\dots,1)^T\in\RR^r$, $e_1$ is the first unit vector (of dimension $s+1$), and
$$\bar\cI_r = \pmatrix{ccc} \int_0^{\hc_1}P_0(x)dx &\dots
&\int_0^{\hc_1}P_{s-1}(x)dx\\
\vdots& & \vdots\\ \int_0^{\hc_r}P_0(x)dx &\dots
&\int_0^{\hc_r}P_{s-1}(x)dx\endpmatrix\in\RR^{r\times s}.$$

The following result generalizes Theorem~\ref{ordth} to the
present setting  (the proof being similar).

\begin{theo}\label{ordth1} Each row of the block method
(\ref{hbvmks})--(\ref{hbvmks1}) defines a LMF of order at least
$s$. The $s$-th row corresponds to the Lobatto quadrature formula
of order $2k$.
\end{theo}

\begin{defn} We call the method defined by the pencil $(A,B)$ in
(\ref{hbvmks})--(\ref{hbvmks1}) a {\em ``Hamiltonian BVM with $k$
steps and degree $s$''}, hereafter {\em
HBVM($k$,$s$)}.\footnote{\,Indeed, the pencil $(A,B)$ perfectly
fits the framework of block BVMs (see, e.g., \cite{BT1}).}
\end{defn}

\begin{rem}\label{snotk}
The structure of the nonlinear system associated with the
HBVM$(k,s)$ is better visualized by performing a permutation of
the stages that splits, into two block sub-vectors, the
fundamental stages and the silent ones. More precisely, the
permuted vector of stages, say $\bfz$, is required to be:
$$\begin{array}{l}
\bfz =
[\underbrace{y_0^T,y_1^T,\dots,y_s^T},\underbrace{w_1^T,w_2^T,\dots,w_r^T}]^T
\equiv [y_0^T, \bfy^T,\bfw^T]^T. \\
\hspace*{.5cm} \mbox{\footnotesize fundamental stages} \hspace*{.5cm}
\mbox{\footnotesize silent stages}
\end{array}
$$
This is accomplished by introducing the permutation matrices $W\in\RR^{k\times
k}$ and\, $W_1\in\RR^{k+1\times k+1}$, such that $$W\pmatrix{c}2\\ \vdots\\
k+1\endpmatrix = \pmatrix{c} \ind_s(2:s+1)\\ \ind_r\endpmatrix,\qquad
W_1\pmatrix{c}1\\ \vdots\\ k+1\endpmatrix =
\pmatrix{c} \ind_s\\ \ind_r\endpmatrix.$$
It is easy to realize that
$$W\,A\,W_1^T = \pmatrix{ccc} -e &I_s &0_{s\times r}\\  -a_0 &-A_1
&I_r\endpmatrix,
\qquad W\,B\,W_1^T = \pmatrix{ccc} b_0 &B_1 &B_2\\ \bf0 & 0_{r\times s} & 0_{r
\times r}\endpmatrix,$$
where  $[-a_0, \, -A_1]$ coincides with $A(s+1:k,\ind_s)$ in \eqref{hbvmks1},
while $[b_0, \, B_1, \, B_2]$ matches the matrix $B(1:s,:)$ in \eqref{hbvmks}.
The HBVM$(k,s)$ then takes  the form:
\begin{equation}
\label{hbvm_sep}
\pmatrix{ccc} -e &I_s &0_{s\times r}\\  -a_0 &-A_1 &I_r\endpmatrix \otimes
I_{2m}\, \bfz = h \pmatrix{ccc} b_0 &B_1 &B_2\\ \bf0 & 0_{r\times s} & 0_{r
\times r}\endpmatrix \otimes J\, \nabla H(\bfz).
\end{equation}
The presence of the null blocks in the lower part of $W\,B\,W_1^T$ clearly
suggests that the (generally nonlinear) system \eqref{hbvm_sep} of (block) size
$k$ is actually equivalent to a system having (block) size $s$. Indeed, we can
easily remove the silent stages, $$\bfw = a_0 \otimes y_0+A_1\otimes I_{2
m}\,\bfy,$$ and obtain
\begin{eqnarray}
\label{hbvm_final}
\bfy &=& e \otimes y_0 +h b_0 \otimes (J \nabla H(y_0))+h B_1 \otimes J\,
\nabla H(\bfy)\\ \nonumber
&&  \qquad \qquad + h B_2 \otimes J\, \nabla H(a_0 \otimes y_0+A_1\otimes
I_{2m}\,\bfy).
\end{eqnarray}
We refer to \cite{BIS} for an alternative technique to reduce the dimension of system \eqref{hbvm_sep}.
\end{rem}

\begin{rem}\label{RK} As was shown in the previous section, when $k=s$, the
HBVM($s$,$s$) coincides with the Lobatto IIIA method of order $2s$.
More in general, for $k\ge s$, by summing up (\ref{hbvmks})--(\ref{hbvmks1}), we
can cast HBVM$(k,s)$ as a Runge-Kutta method with the following tableau:
\begin{equation}\label{tableau}
\begin{array}{c|c}
t_0    &\\
\vdots &  \bar\cI D_s \bar\P^T\bar\Omega\\
t_k    &\\
\hline
       & b_0\quad\dots\quad b_k\end{array}
\end{equation}
where $$\bar\cI = \pmatrix{ccc} \int_0^{t_0}P_0(x)dx &\dots
&\int_0^{t_0}P_{s-1}(x)dx\\
\vdots& & \vdots\\ \int_0^{t_k}P_0(x)dx &\dots
&\int_0^{t_k}P_{s-1}(x)dx\endpmatrix\in\RR^{k+1\times s}.$$ We
observe that the $k+1\times k+1$ matrix
\begin{equation}\label{C}
C = \bar\cI D_s \bar\P^T\bar\Omega \end{equation} appearing in
(\ref{tableau}) has rank $s$, thus confirming that the
computational cost per iteration depends on $s$, rather than on
$k$ (see \cite{BIS} for more details and a practical example of
Butcher tableau concerning the method HBVM(6,2)).

By the way, we observe that, when $s=1$, HBVM$(k,1)$ are nothing
but the ``$s$-stage trapezoidal methods'', defined in \cite{IP1},
based on the Lobatto abcissae. In such a case, the matrix $C$
becomes
$$C = \pmatrix{c}t_0\\ \vdots \\ t_k\endpmatrix \pmatrix{ccc} b_0 &
\dots &b_k\endpmatrix.$$ Similarly, for $s=2$ and $k=4$, HBVM(4,2)
coincides with the fourth-order method presented in
\cite[Section\,4.2]{IT2}, able to preserve polynomial Hamiltonians
of degree four.
\end{rem}

Concerning the order of convergence, the following result generalizes that of
Corollary~\ref{collcor}.

\begin{cor}\label{ordine}
The HBVM($k$,$s$) (\ref{hbvmks})--(\ref{hbvmks1}) has order of
convergence  $p=2s$.
\end{cor}
\proof By virtue of Theorem~\ref{ordth1}, the corresponding Runge-Kutta method
(\ref{tableau}) satisfies the usual simplifying assumptions $B(2k)$ and $C(s)$.
If we are able to prove $D(s-1)$, from the classical result of Buthcher (see,
e.g., \cite[Theorem\,5.1]{HW}), it will follow that the method has order $p=2s$.
With reference to (\ref{tableau}), the condition $D(s-1)$ can be cast in matrix
form, by introducing the
vectors $e=(1,\dots,1)^T\in\RR^{s-1}$, $\bar{e}=(1,\dots,1)^T\in\RR^{k+1}$, and
the matrices $$Q=\diag(1,\dots,s-1),\qquad T=\diag(t_0,\dots,t_k),\qquad
V=(t_{i-1}^{j-1})\in\RR^{k+1\times s-1},$$ as
$$Q V^T\bar\Omega\left(\bar\cI D_s \bar\P^T\bar\Omega\right) =
\left(e\,\bar{e}^T -V^TT\right)\bar\Omega,$$ i.e.,

\begin{equation}\label{finito}
\bar\P D_s\bar\cI^T\bar\Omega V Q = \left(\bar{e}\,e^T -TV\right).
\end{equation}
Since the quadrature is exact for polynomials of degree $2s-1 \le 2k-1$, one has

\begin{eqnarray*}
\left(\bar\cI^T\bar\Omega VQ\right)_{ij} &=& \left( \sum_{\ell=0}^k b_\ell
\int_0^{t_\ell} P_{i-1}(x)\mathrm{d}x\,(j t_\ell^{j-1}) \right) =
\left(\int_0^1 \, \int_0^t P_{i-1}(x)\mathrm{d}x
(jt^{j-1})\mathrm{d}t\right)
\\&=& \left( \delta_{i1}-\int_0^1P_{i-1}(x)x^j\mathrm{d}x\right),
\qquad i = 1,\dots,s,\quad j=1,\dots,s-1,\end{eqnarray*} where the
last equality is obtained by integrating by parts, with
$\delta_{i1}$ the Kronecker symbol. Consequently,
\begin{eqnarray*}
\left(\bar\P D_s\bar\cI^T \bar\Omega V Q\right)_{ij} &=& \left(1 -
\sum_{\ell=0}^{s-1} \eta_\ell P_\ell(t_i)\int_0^1 P_\ell(x)
x^j\mathrm{d}x \right) = (1-t_{i-1}^j),\\&& \qquad
\qquad
i=1,\dots,k+1,\quad j=1,\dots,s-1,\end{eqnarray*} that is,
(\ref{finito}), where the last equality follows from the fact that
$$\sum_{\ell=0}^{s-1} \eta_\ell P_\ell(t)\int_0^1 P_\ell(x)
x^j\mathrm{d}x = t^j, \qquad j=1,\dots,s-1.\QED$$ \medskip

An additional, remarkable property of such methods is gained,
provided that the abscissae $\{t_0,\dots,t_k\}$ (\ref{ti}) are
symmetrically distributed (as is the case of the Lobatto abscissae
here considered). For this purpose, we need to introduce some
notations and preliminary results. Let us define the matrix
$$E_n =
\pmatrix{ccccc}&&&&1\\&&&\cdot\\&&\cdot\\&\cdot\\1\endpmatrix\in\RR^{n\times
n},$$

\no which, when applied to a vector of length $n$, reverses the
order of its entries. We also set \begin{equation}\label{LF}
L = \pmatrix{rrrr} -1&1\\
&\ddots&\ddots\\&&-1&1\endpmatrix\in \RR^{k\times k+1}, \quad F = \pmatrix{cccc} (-1)^0 \\
&(-1)^1\\&&\ddots\\&&&(-1)^{s-1}\endpmatrix\in \RR^{s\times
s}.\end{equation}

\no The following preliminary result holds true.

\begin{lem}\label{sim1} If the abscissae (\ref{ti}) are symmetric,
then matrix (\ref{C}) satisfies:
$$E_k L\, C \, E_{k+1} = L\,C.$$
\end{lem}

\proof From the symmetry of the abscissae it easily follows that
(see (\ref{bik}) and (\ref{barOm})) $$E_{k+1} \bar\Omega E_{k+1} =
\bar\Omega.$$ From property {\bf P6}, we have that (see
(\ref{bps1})) $$\bar\P^T E_{k+1} = F\,\bar\P^T.$$

\no Moreover, by considering that (see (\ref{ci}))$$L\,\cI =
\pmatrix{ccc} \int_{t_0}^{t_1} P_0(x)dx & \dots &
\int_{t_0}^{t_1} P_{s-1}(x)dx\\ \vdots & &\vdots\\
\int_{t_{k-1}}^{t_k} P_0(x)dx & \dots & \int_{t_{k-1}}^{t_k}
P_{s-1}(x)dx\endpmatrix,$$ again from {\bf P6} we see that
$$E_s L\,\cI = L\,\cI F.$$ Finally, from (\ref{C})
we obtain
\begin{eqnarray*}
\lefteqn{E_k L\, C\, E_{k+1} =}\\ &=& (E_k L\,\cI)
D_s(\bar\P^TE_{k+1})(E_{k+1}\bar\Omega\,E_{k+1}) \\
&=& L\,\cI F\, D_s F \bar\P^T\bar\Omega = L\,\cI D_s
\bar\P^T\bar\Omega = L\, C.~\QED
\end{eqnarray*}

As a consequence, we have the following result.

\begin{theo}\label{simmetrico}
If the abscissae (\ref{ti}) are symmetric, then the method
(\ref{hbvmks})--(\ref{hbvmks1}) (i.e., (\ref{tableau})) is
symmetric, that is, it is self-adjoint.
\end{theo}

\proof Indeed, the discrete solution, $\hbfy$, satisfies the
equation (see (\ref{tableau})--(\ref{C}) and (\ref{LF}))
$$L\otimes I_{2m}\, \hbfy = h\,L\,C\otimes I_{2m}\,f(\hbfy).$$

\no Considering that $E_kLE_{k+1}=-L$ and, from Lemma~\ref{sim1},
$E_k\,L\,C\,E_{k+1}=LC$, one then obtains
\begin{eqnarray*}L\otimes I_{2m} (E_{k+1}\otimes I_{2m}\,\hbfy)
&=&-h\,L\,C\otimes I_{2m}\left(E_{k+1}\otimes I_{2m}\,f(\hbfy)\right)\\&=&
-h\,L\,C\otimes I_{2m}\, f\left( E_{k+1}\otimes
I_{2m}\,\hbfy\right).\end{eqnarray*} \no The thesis then
follows by observing that the vector $E_{k+1}\otimes I_{2m}\,
\hbfy$ contains the time-reversed discrete solution.\QED

\medskip The next theorem summarizes the results about HBVM$(k,s)$.

\begin{theo}[Main Result]\label{main}
For all $s=1,2,\dots$, and $k\ge s$, the HBVM($k$,$s$) method:
\begin{enumerate}
\item has order of accuracy $2s$;

\smallskip
\item is energy-preserving for polynomial Hamiltonians of degree
not larger than $2k/s$;

\smallskip
\item for general $C^{(2k+1)}$ Hamiltonians, the energy error at each
integration step is $O(h^{2k+1})$, if $h$ is the used
stepsize;\,\footnote{Consequently, on any finite interval the
global energy error is not larger than $O(h^{2k})$.}

\smallskip
\item is symmetric and, therefore, precisely $A$-stable.
\end{enumerate}
\end{theo}

\begin{proof} Item 1 follows from  Corollary~\ref{ordine}.
 Item 2 follows from the fact that, for such polynomial Hamiltonians, the
vanishing discrete line integral equals the continuous line
integral (see (\ref{exactness}) and (\ref{intexk})). Similarly,
Item 3 follows from the fact that, by using arguments similar to
those used in Remark~\ref{Rs} (see (\ref{intex})), the energy
error per integration step equals the quadrature error of the
Gauss-Lobatto formula of order $2k$. Finally, Item 4 follows from
Theorem~\ref{simmetrico}, since the Lobatto abscissae $\{t_i\}$
are symmetrically distributed.\QED\end{proof}

\begin{rem} From the result of Theorem~\ref{main}, we can then concude
that HBVM$(k,s)$ is optimal, both from the point of view of the order and
stability properties. Moreover, its computational cost, as observed in
Remarks~\ref{snotk} and \ref{RK}, is seen to depend on $s$, rather than on $k$.
\end{rem}

\section{Numerical Tests}\label{tests}
We here report a few numerical tests, in order to show the potentialities of
HBVM$(k,s)$.

Let then consider, at first, the Hamiltonian problem characterized
by the polynomial Hamiltonian (4.1) in \cite{Faou},
\begin{equation}\label{fhp}
H(p,q) = \frac{p^3}3 -\frac{p}2 +\frac{q^6}{30} +\frac{q^4}4
-\frac{q^3}3 +\frac{1}6,
\end{equation}

\no having degree $\nu=6$, starting at the initial point
$y_0\equiv (q(0),p(0))^T=(0,1)^T$. For such a problem, in
\cite{Faou} it has been  experienced a numerical drift in the
discrete Hamiltonian, when using the fourth-order Lobatto IIIA
method\,\footnote{\,Such method coincides with the HBVM(2,2) above
described.} with stepsize $h=0.16$. This is confirmed by the plot
in Figure~\ref{faoufig}, where a linear drift in the numerical
Hamiltonian is clearly observable. On the other hand, by using the
fourth-order HBVM(6,2)  with the same stepsize, the drift
disappears, as shown in Figure~\ref{faoufig1}, since such method
exactly preserves polynomial Hamiltonians of degree up to 6.
Moreover, the order of convergence $p=4$ is (numerically)
confirmed by the results listed in Table~\ref{tab1}, where the
used stepsizes $h$, the maximum estimated error (obtained as the
difference of two consecutive solutions), and the estimated order
of convergence are listed.

\begin{figure}[hp]
\centerline{\includegraphics[width=0.9\textwidth]{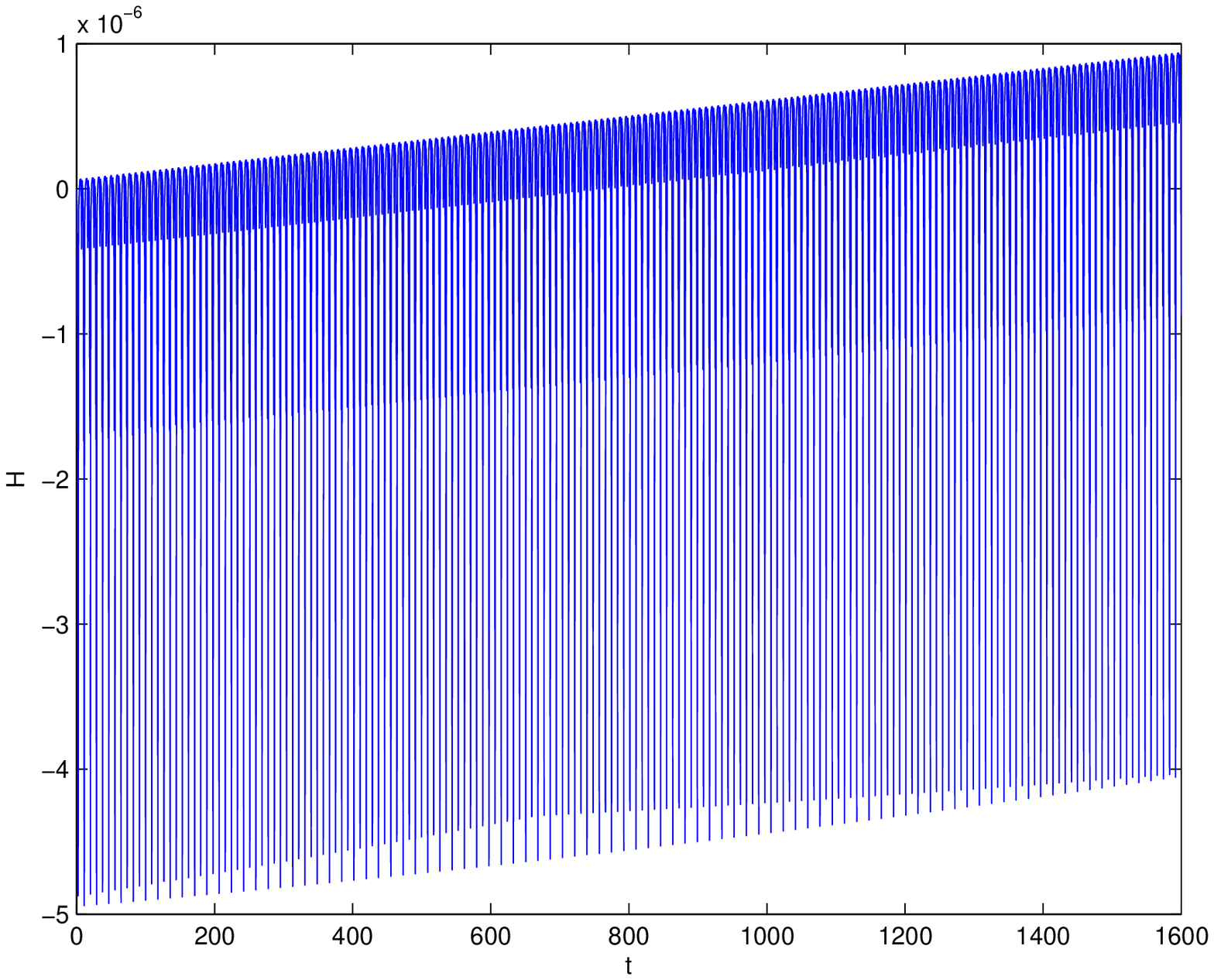}}
\caption{\protect\label{faoufig} Fourth-order Lobatto IIIA method,
$h=0.16$, problem (\ref{fhp}).} \bigskip

\centerline{\includegraphics[width=0.9\textwidth]{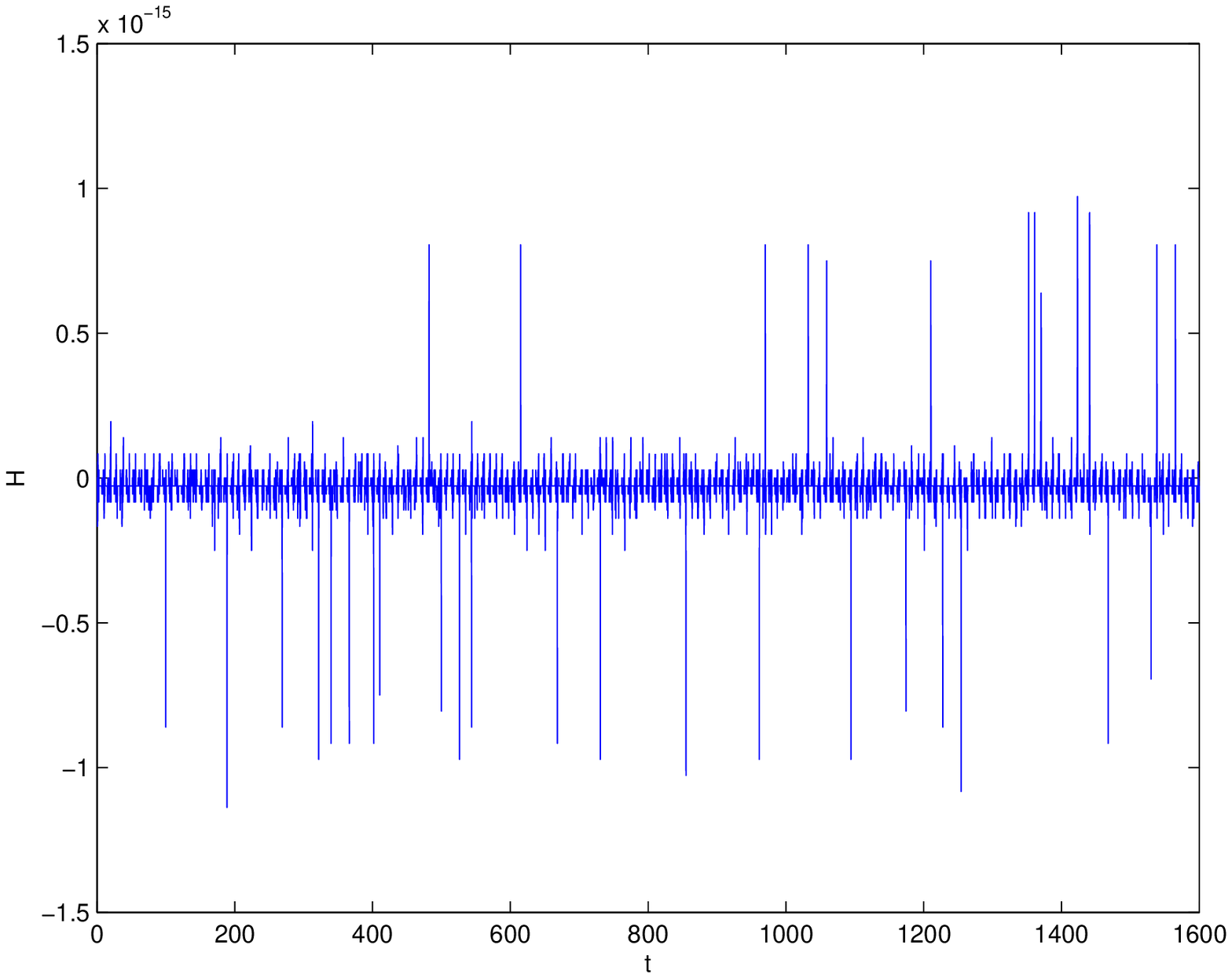}}
\caption{\protect\label{faoufig1} Fourth-order HBVM(6,2) method,
$h=0.16$, problem (\ref{fhp}).}
\end{figure}

The second test problem is the Fermi-Pasta-Ulam problem (see
\cite[Section\,I.5.1]{HLW}), defined by the Hamiltonian
\begin{equation}\label{fpu}
H(p,q) = \frac{1}2\sum_{i=1}^m\left(p_{2i-1}^2+p_{2i}^2\right)
+\frac{\omega^2}4\sum_{i=1}^m\left(q_{2i}-q_{2i-1}\right)^2
+\sum_{i=0}^m\left(q_{2i+1}-q_{2i}\right)^4,
\end{equation}

\no with $q_0=q_{2m+1}=0$, $m=3$, $\omega=50$, and starting vector
$$p_i=0, \quad q_i = (i-1)/10, \qquad i=1,\dots,6.$$ In such a
case, the Hamiltonian function is a polynomial of degree 4, so
that the fourth-order HBVM(4,2) method, which is used with
stepsize $h=0.05$, is able to exactly preserve the Hamiltonian, as
confirmed by the plot in Figure~\ref{fpufig1}, whereas the
fourth-order Lobatto IIIA method provides the result plotted in
Figure~\ref{fpufig}. Moreover, in Table~\ref{tab2} we list
corresponding results as in Table~\ref{tab1}, again confirming the
fourth-order convergence.

\begin{figure}[hp]
\includegraphics[width=0.9\textwidth]{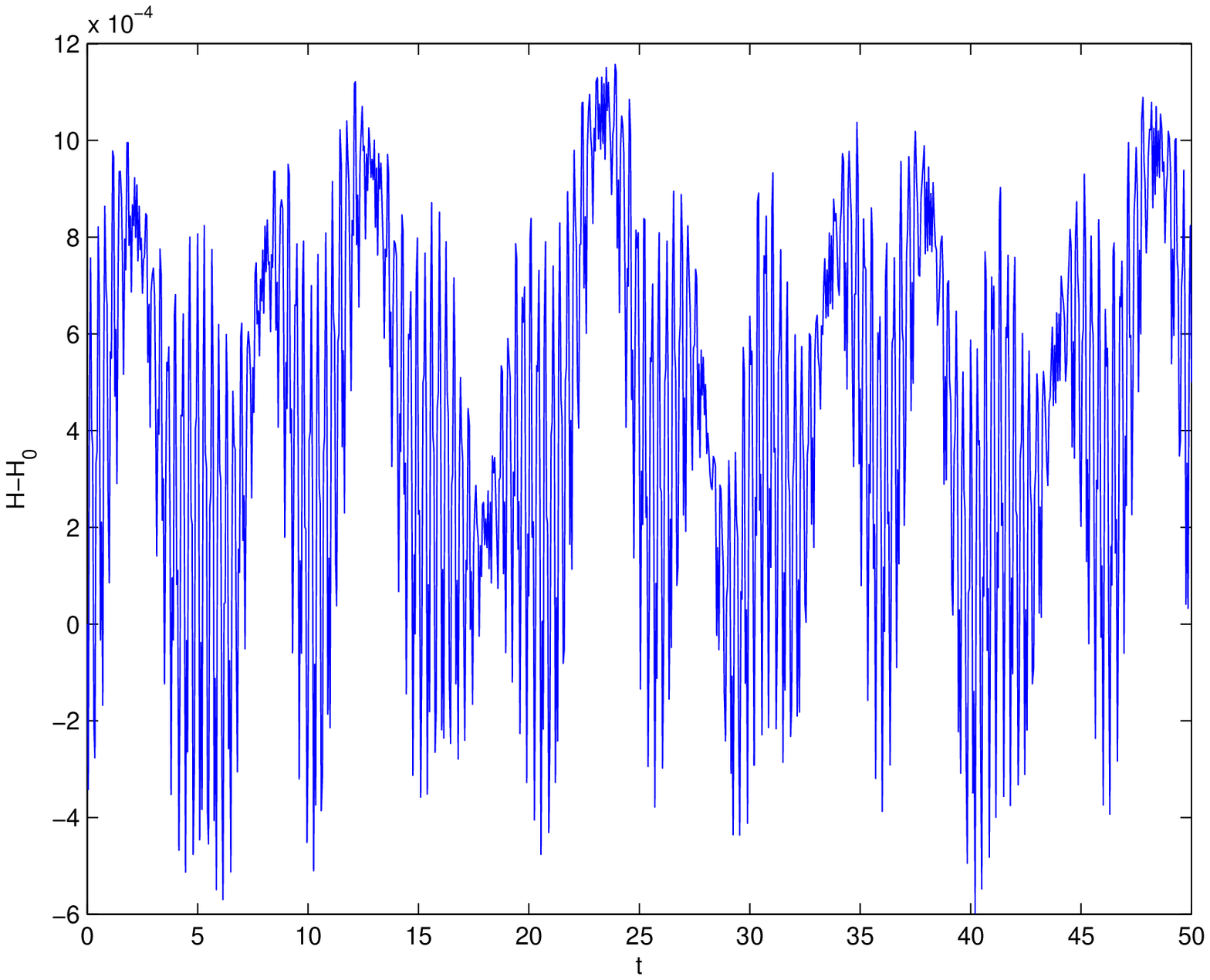}
\caption{\protect\label{fpufig} Fourth-order Lobatto IIIA method,
$h=0.05$, problem (\ref{fpu}).}\bigskip

\includegraphics[width=0.9\textwidth]{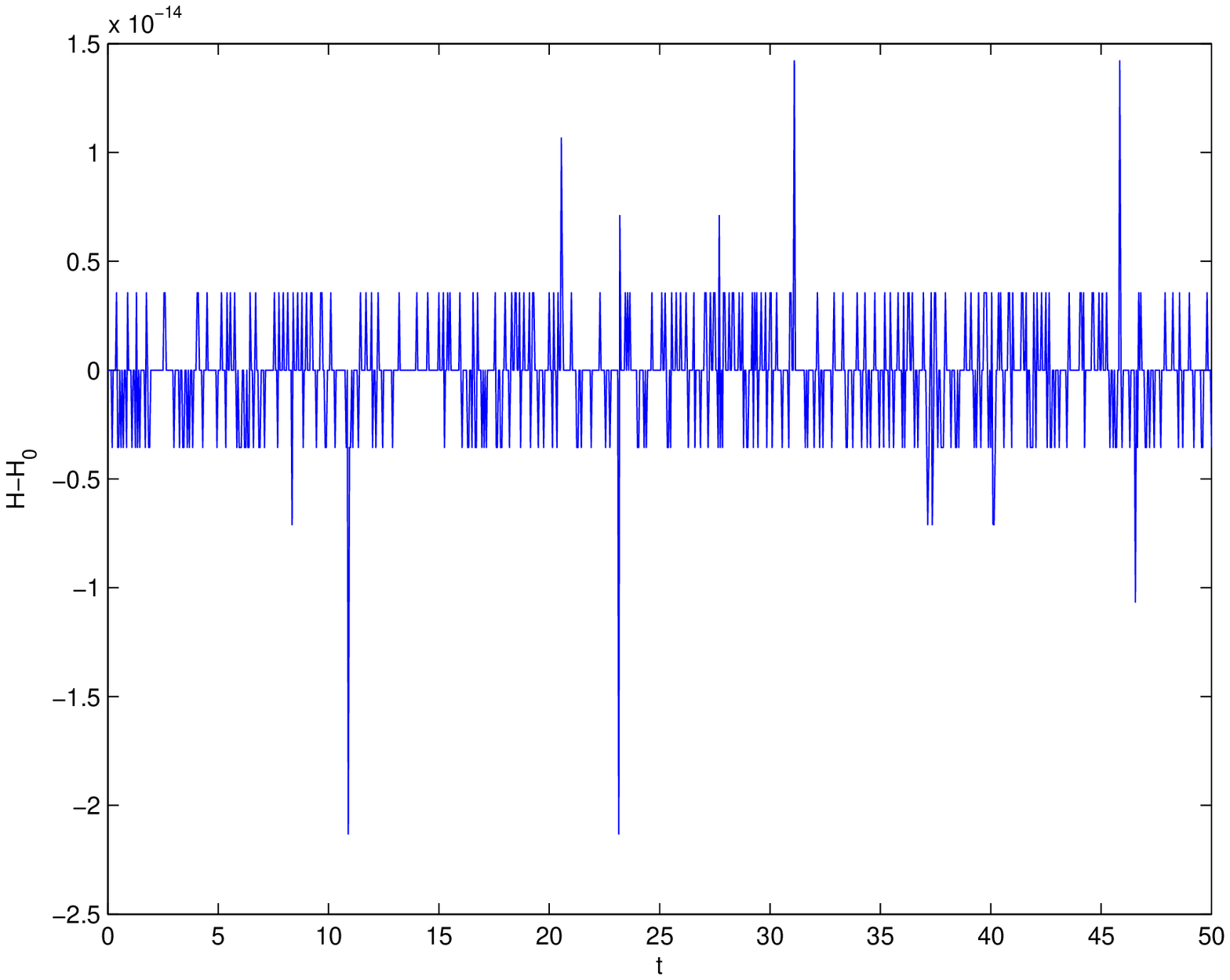}
\caption{\protect\label{fpufig1} Fourth-order HBVM(4,2) method,
$h=0.05$, problem (\ref{fpu}).}
\end{figure}

\begin{table}[t]
\caption{\protect\label{tab1} Numerical order of convergence for
the HBVM(6,2) method, problem (\ref{fhp}).}
\begin{tabular}{|c|lllll|} \hline
$h$  & 0.32 & 0.16 & 0.08 & 0.04 & 0.02 \\
\hline
error  &$2.288\cdot10^{-2}$ &$1.487\cdot10^{-3}$ &$9.398\cdot10^{-5}$ &$5.890\cdot10^{-6}$ &$3.684\cdot10^{-7}$\\
\hline
order  & -- & 3.94 &3.98 &4.00 &4.00\\
\hline
\end{tabular}
\medskip

\caption{\protect\label{tab2} Numerical order of convergence for
the HBVM(4,2) method, problem (\ref{fpu}).}
\begin{tabular}{|c|lllll|} \hline
$h$  & $1.6\cdot10^{-2}$ &$8\cdot10^{-3}$ & $4\cdot10^{-3}$ & $2\cdot10^{-3}$ & $10^{-3}$ \\
\hline
error  & $3.030$ &$1.967\cdot10^{-1}$ &$1.240\cdot10^{-2}$ &$7.761\cdot10^{-4}$ &$4.853\cdot10^{-5}$ \\
\hline
order  & -- & 3.97& 3.99 &4.00 &4.00\\
\hline
\end{tabular}
\medskip

\caption{\protect\label{tab3} Numerical order of convergence for
the HBVM(6,2) method, problem (\ref{biot}).}
\begin{tabular}{|c|lllll|} \hline
$h$  & $3.2\cdot10^{-2}$ &$1.6\cdot10^{-2}$ &$8\cdot10^{-3}$ & $4\cdot10^{-3}$ & $2\cdot10^{-3}$ \\
\hline
error  & $3.944\cdot10^{-6}$ &$2.635\cdot10^{-7}$ &$1.729\cdot10^{-8}$
&$1.094\cdot10^{-9}$ &$6.838\cdot10^{-11}$ \\
\hline
order  & -- & 3.90& 3.93 &3.98 &4.00\\
\hline
\end{tabular}
\end{table}

In the previous examples, the Hamiltonian function was a
polynomial. Nevertheless, as is easily argued from
Theorem~\ref{main}, HBVM($k$,$s$) are expected to produce  a {\em
practical} conservation of the energy when applied to systems
defined by a non-polynomial Hamiltonian function which are
sufficiently differentiable. As an example, we consider the motion
of a charged particle in a magnetic field with Biot-Savart
potential.\footnote{\,As an example, this kind of motion causes
the well known phenomenon of {\em aurora borealis}.} It is defined
by the Hamiltonian
\begin{eqnarray}\label{biot}
\lefteqn{H(x,y,z,\dot{x},\dot{y},\dot{z}) = }\\&&\frac{1}{2m}
\left[ \left(\dot{x}-\aa\frac{x}{\rho^2}\right)^2 +
\left(\dot{y}-\aa\frac{y}{\rho^2}\right)^2 +
\left(\dot{z}+\aa\log(\rho)\right)^2\right],\nonumber
\end{eqnarray}

\no with $\rho=\sqrt{x^2+y^2}$, $\aa= e \,B_0$,  $m$ is the
particle mass, $e$ is its charge, and $B_0$ is the magnetic field
intensity. We have used the values $$m=1, \qquad e=-1, \qquad
B_0=1,$$with starting point
$$x = 0.5, \quad y = 10, \quad z = 0, \quad
\dot{x} =  -0.1, \quad \dot{y} = -0.3, \quad \dot{z} = 0.$$

\no In Figure~\ref{biotfig0}, the trajectory of the particle in
the interval $[0,10^3]$ is plotted in the phase space. As one can
see, it is a helix that wings downward. By using the fourth-order
Lobatto IIIA method with stepsize $h=0.1$, a drift in the
numerical Hamiltonian can be again observed (see
Figure~\ref{biotfig}),  so that the method does introduce a
friction. When using the HBVM(4,2) method with the same stepsize,
the drift disappears and the Hamiltonian turns out to be almost
preserved (see Figure~\ref{biotfig1}). As expected, the result
improves if we increase $k$: the plot in Figure~\ref{biotfig2} has
been obtained by using the HBVM(6,2), from which one realizes that
a practical preservation of the Hamiltonian is reached. Finally,
the data listed in Table~\ref{tab3} confirm the fourth-order
convergence of the latter method.

\begin{figure}[hp]
\includegraphics[width=0.9\textwidth]{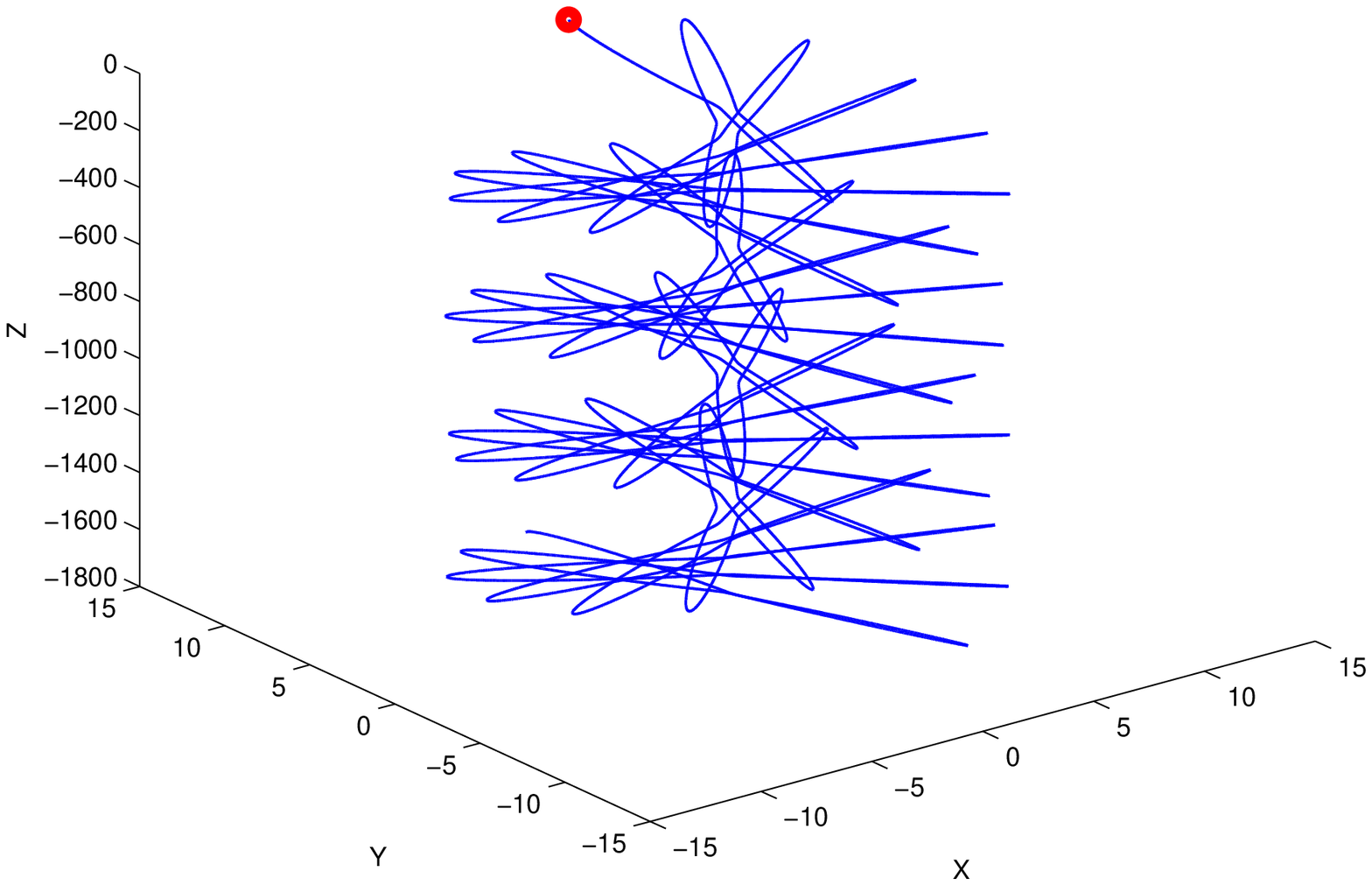}
\caption{\protect\label{biotfig0} Phase-space plot of the solution of problem
(\ref{biot}) for $0\le t\le 10^3$ (the circle denotes the starting
point of the trajectory).}\bigskip

\includegraphics[width=0.9\textwidth]{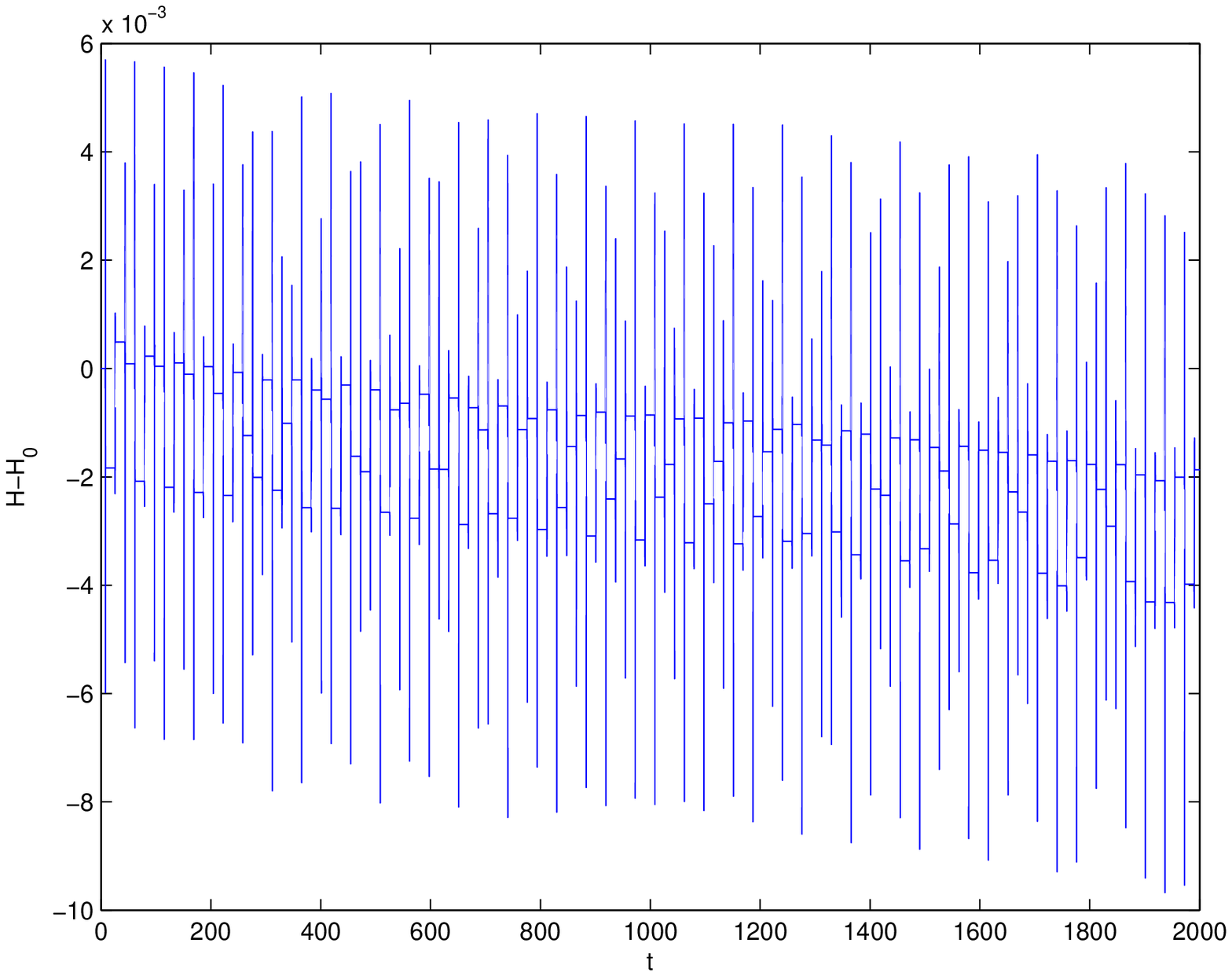}
\caption{\protect\label{biotfig} Fourth-order Lobatto IIIA method,
$h=0.1$, problem (\ref{biot}).}
\end{figure}

\begin{figure}[hp]
\includegraphics[width=0.9\textwidth]{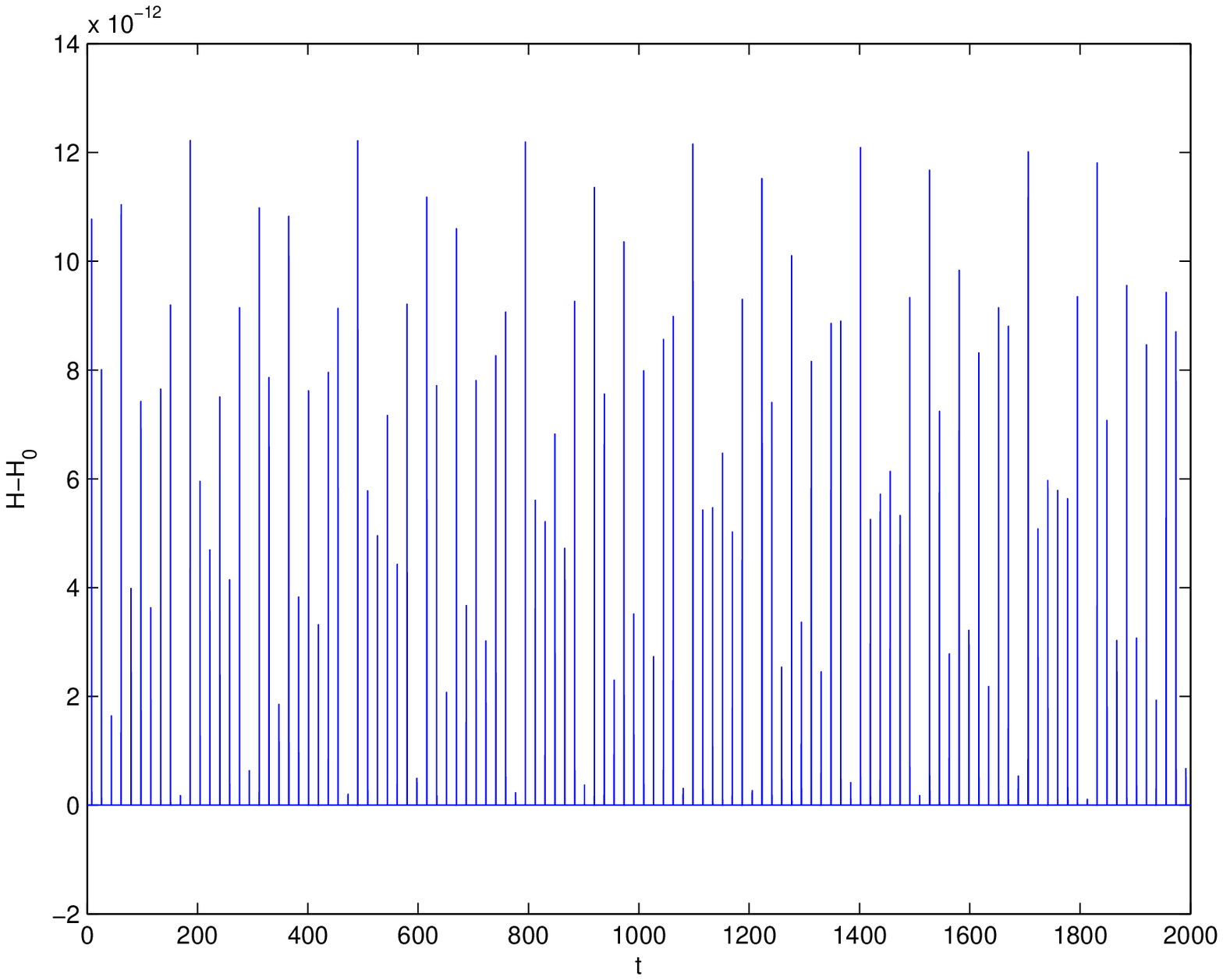}
\caption{\protect\label{biotfig1} Fourth-order HBVM(4,2) method,
$h=0.1$, problem (\ref{biot}).}\bigskip

\includegraphics[width=0.9\textwidth]{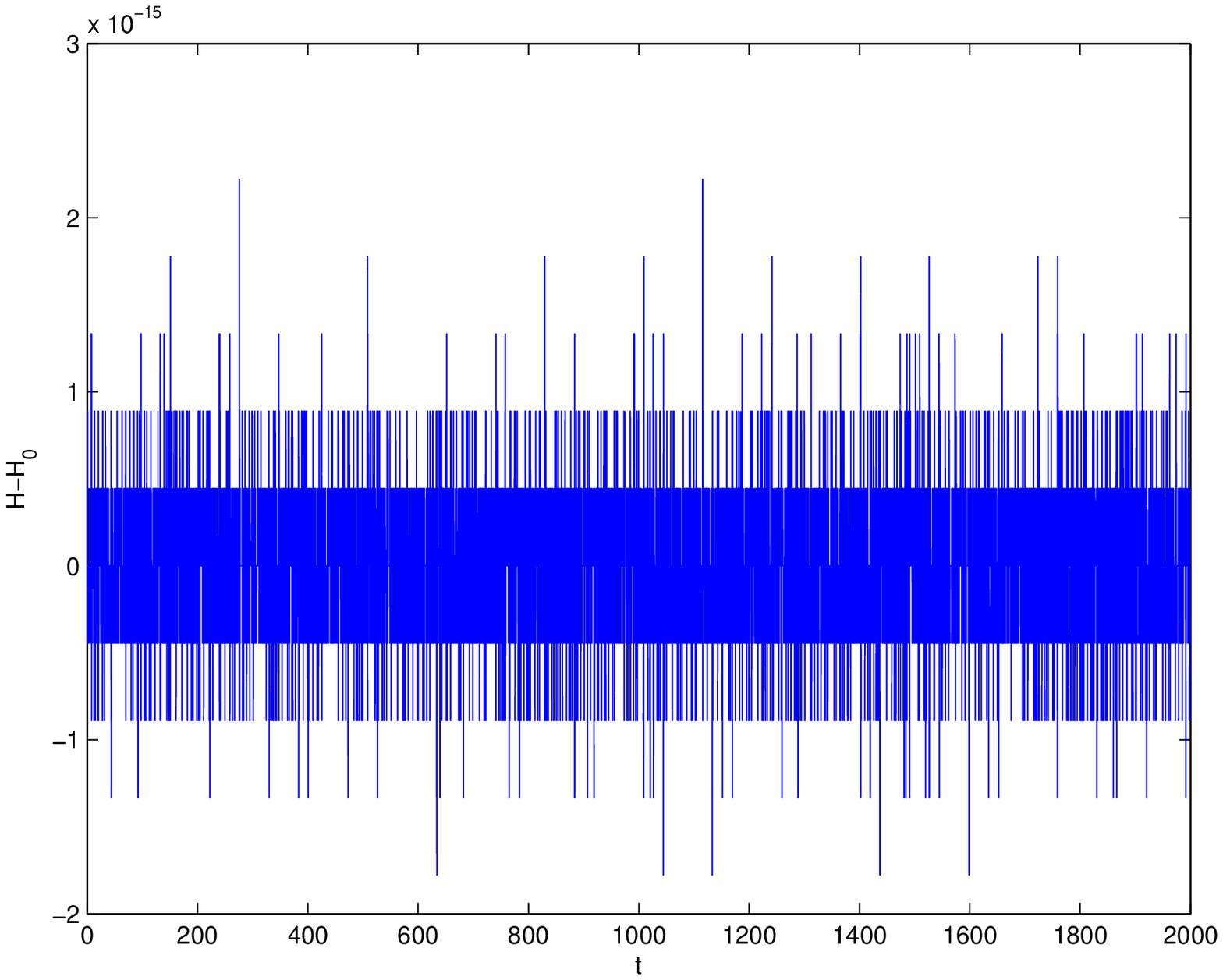}
\caption{\protect\label{biotfig2} Fourth-order HBVM(6,2) method,
$h=0.1$, problem (\ref{biot}).}
\end{figure}

\section{Conclusions}\label{fine} In this paper a new class of numerical
methods, able to preserve polynomial Hamiltonians, has been
studied in details. From the analysis, it turns out that such
methods can be regarded as a generalization of collocation
Runge-Kutta Lobatto IIIA methods. Nevertheless, the fact of being
characterized by a matrix pencil, perfectly fits the framework of
block BVMs, so that we have named them Hamiltonian BVMs (HBVMs). A
number of numerical tests prove their effectiveness in preserving
the Hamiltonian function when evaluated along the numerical
solution,  as well as confirm the predicted order of convergence.
Possible different choices of the abscissae, as well as the actual
efficient implementation of the methods, will be the subject of
future investigations.

\section*{Acknowledgements} We wish to thank
prof.\,Mario Trigiante for providing us with problem (\ref{biot}).
We are grateful to Ernst Hairer, for his comments on the proof of
Corollary~\ref{ordine}. We also thank the referees, for their useful
comments and suggestions.

\section*{Appendix: some properties of shifted Legendre polynomials}

A number of useful properties of shifted Legendre polynomials are
here summarized: for their proof see any book on special functions
(e.g., \cite{AS}).
\begin{description}

\item[\bf P1.] Generalized Rodrigues formula: for all $n=0,1,\dots$,
\,$P_n(x)$ has degree $n$ and can be defined as $$P_n(x) =
\frac{1}{n!} \frac{d^n}{dx^n} \left[(x^2-x)^n\right].$$

\item[\bf P2.] Lobatto quadrature: the Lobatto abscissae $\{c_i\}$ (\ref{ci}),
of the formula of degree $2s$, are the zeros of the polynomial
$$(x^2-x) P_s'(x),$$ where $P'_s(x)$ denotes the
derivative of $P_s(x)$. The corresponding weights (\ref{bi}) are
given by:
$$b_i = \frac{1}{s(s+1) (P_s(c_i))^2}, \qquad i=0,1,\dots,s,$$
which are, therefore, all positive.

\item[\bf P3.] Orthogonality: $$\int_0^1 P_n(x)P_m(x)\,dx =
\frac{1}{2n+1} \delta_{nm}, \qquad n=0,1,\dots,$$ where, as usual,
$\delta_{nm}$ denotes the Kronecker delta.

\item[\bf P4.] Recurrence formula: by setting hereafter $P_{-1}(x)\equiv 0$ and
$P_0(x)\equiv 1$,
$$(n+1)P_{n+1}(x) = (2n+1)(2x-1)P_n(x)-n P_{n-1}(x), \qquad
n=0,1,\dots.$$

\item[\bf P5.] Explicit formula: $$P_n(x) = (-1)^n\sum_{i=0}^n \pmatrix{c} n\\
i\endpmatrix \pmatrix{c} n+i\\ i\endpmatrix (-x)^i, \qquad
n=0,1,\dots.$$

\item[\bf P6.] Symmetry: $$P_n(1-x) = (-1)^n P_n(x), \qquad
n=0,1,\dots.$$

\item[\bf P7.] Symmetry at the end-points: $$P_n(0) = (-1)^n, \quad P_n(1) = 1, \qquad n=0,1,\dots.$$

\item[\bf P8.] Derivatives: $$ 2(2n+1) P_n(x) = \frac{d}{dx} \left[ P_{n+1}(x) -
P_{n-1}(x) \right], \qquad n=0,1,\dots.$$

\item[\bf P9.] Integrals: \begin{eqnarray*}2\int_{0}^x P_0(t)\,dt &=&
2x = P_1(x)+P_0(x),\\[2mm]
2(2n+1)\int_{0}^x P_n(t)\,dt &=& P_{n+1}(x) - P_{n-1}(x), \qquad
n=1,2,\dots.\end{eqnarray*}

\item[\bf P10.] Shifted Legendre differential equations. The shifted
Legendre polynomials satisfy the second order differential
equation:
$$\frac{d}{dx} \left[(x^2-x)P_n'(x)\right] + n(n+1)P_n(x) = 0, \qquad
n=0,1,\dots.$$

\item[\bf P11.] From {\bf P2} and {\bf P10}, it follows that, if
(\ref{ci}) are the Lobatto abscissae of the formula of order $2s$
(i.e., exact for polynomials of degree $2s-1$), then
$$\int_0^{c_i} P_s(x)\,dx = 0, \qquad i=0,1,\dots,s.$$

\item[\bf P12.] A few examples:
\begin{eqnarray*}
P_0(x) &\equiv& 1,\\
P_1(x) &=& 2x-1,\\
P_2(x) &=& 6x^2-6x+1,\\
P_3(x) &=& 20x^3-30x^2+12x-1,\\
\dots
\end{eqnarray*}

\end{description}
%%\bigskip (N.B.: LE PROPRIETA' P1, P4, P5, P8, P12 NON SONO ESPLICITAMENTE
%%UTILIZZATE).

\end{document}